\documentclass[a4paper, 12pt]{amsart}

\usepackage{amsmath}
\usepackage{amssymb}
\usepackage{amsthm}
\usepackage[usenames,dvipsnames]{xcolor}

\usepackage[final]{showkeys}

\usepackage{url}
\usepackage{ mathrsfs }

\definecolor{DeepBlue}{rgb}{0,0,.7}
\usepackage[dvipdfm, setpagesize=false]{hyperref}
\hypersetup{colorlinks=true, citecolor=Sepia, linkcolor=DeepBlue, urlcolor=DeepBlue, pdfborder= 0 0 0}
\hypersetup{setpagesize=false}
\usepackage[all]{hypcap}
\usepackage[all]{xy}

\newcommand{\HH}{\mathbb{H}}
\newcommand{\SO}{\textnormal{SO}}
\newcommand{\PSL}{\textnormal{PSL}}

\newcommand{\R}{\mathbb{R}}

\newcommand{\Ob}{\mathcal{O}}
\newcommand{\C}{\mathbb{C}}
\newcommand{\Z}{\mathbb{Z}}

\newcommand{\w}[1]{\widetilde{#1}}

\newcommand{\bs}{\backslash}

\newcommand{\al}{\alpha}

\newcommand{\Id}{\text{\textnormal{Id}}}

\newcommand{\End}{\text{End}}
\newcommand{\tr}{\text{\textnormal{tr}}}
\newcommand{\diag}{\text{\textnormal{diag}}\thinspace}
\newcommand{\Tr}{\text{\textnormal{Tr}}}

\newcommand{\vol}{\text{\textnormal{vol}}}

\newcommand{\Ad}{\text{Ad}}
\newcommand{\RRe}{\text{Re}}

\newcommand{\N}{\mathbb{N}}
\newcommand{\bsm}{\left( \begin{smallmatrix}}
\newcommand{\esm}{\end{smallmatrix} \right)}

\usepackage{accents}

\newtheorem{Th}{Theorem}[section]
\newtheorem{Lemma}[Th]{Lemma}
\newtheorem{Proposition}[Th]{Proposition}
\newtheorem{Corollary}[Th]{Corollary}
\theoremstyle{definition}
\newtheorem{lem}{Lemma}[section]
\newtheorem{remark}[lem]{Remark}
\newtheorem{definition}{Definition}[section]

\newtheorem{exmp}{Example}[section]

\begin{document}

\title{Analytic torsion of finite volume hyperbolic orbifolds}
\author{Ksenia Fedosova}
\address{Max Planck Institute for Mathematics\\ Vivatgasse 7\\ 53111 Bonn \\ Germany}

\email{fedosova@math.uni-bonn.de}


\begin{abstract}
In this article we define the analytic torsion of finite volume orbifolds $\Gamma \backslash \mathbb{H}^{2n+1}$ and study its asymptotic behavior with respect to certain rays of representations.
\end{abstract}

\maketitle

\setcounter{tocdepth}{1}
\section{Introduction}

The goal of this article is an attack on the problem of growth of torsion in the cohomology of arithmetic groups.
This problem was first studied \cite{MM} in the following setting. Let $\Gamma$ be a cocompact discrete torsion free arithmetic subgroup of $G = \textnormal{SL}_2(\C)$ and let $\textnormal{Sym}^n$ be the $n$-th symmetric power of the standard representation of $\textnormal{SL}_2(\C)$. It was shown that for every even $n$ there exists a lattice $M_n \subset S^n(\C^2)$, the $n$~-~th symmetric power of $\C^2$, which is invariant under $\Gamma$. The main result of \cite{MM} is that the order of the second cohomology $|H^2(\Gamma, M_n)|$ grows exponentially as $n \to \infty$. Later the result was extended \cite{MP111} to the case where $\Gamma$ is a cocompact  torsion free arithmetic subgroup of $G = \textnormal{SO}^0(p,q)$ or $\textnormal{SL}_3(\R)$; 
it was shown that for a certain sequence of arithmetic $\Gamma$-modules $\mathcal{M}_n$ with $n \in \N$ the cohomology $|H^j(\Gamma, \mathcal{M}_n)|$ grows exponentially at least for one $j$.

One of the key ingredients in both proofs is to rewrite the Reidemeister torsion $\tau_X$ in terms of the torsion in the certain cohomology groups as follows. Let $X=\Gamma \bs G / K$, where $G$ and $\Gamma$ are as an above and $K$ is a maximal compact subgroup of $G$; let $\rho_n:G \to GL(V)$ be an irreducible finite dimensional representation of $G$; and let $M_n \subset V$ be a lattice which is stable under $\Gamma$ and denote by $\mathcal{M}_n$ the associated local system of free $Z$-modules over $X$. Then the Reidemeister torsion $\tau_X$ of a manifold $X$ can be written as
$$ \tau_X = R \cdot \prod_{q=0}^n |H^q(X, \mathcal{M}_n)_{tors}|^{(-1)^{q+1}},$$
where the regulator $R$ is expressed in terms of the basis 
 of $H^*(X,\mathcal{M})_{free}$ and the $L^2$-metric on $\mathcal{H}^*(X,E)$. Further one uses the Cheeger-M\"uller theorem to study the growth of the analytic torsion $T(\chi,h)$ instead of the Reidemeister torsion $\tau_X(\chi,h)$. The growth of the analytic torsion as the local system varies has already been studied in different settings: for compact hyperbolic 3-manifolds \cite{Mu2} and odd-dimensional hyperbolic manifolds \cite{MP2} (these references were used in the mentioned \cite{MM, MP111}); compact hyperbolic orbifolds \cite{Fe2}; and finite volume hyperbolic manifolds \cite{MP}. In the last article authors imposed the requirement on $\Gamma$ to be neat in the sense of Definition \ref{neatness}. It turns out that many important arithmetic groups are not neat, for example $\textnormal{SL}(2, \Z \oplus i \Z)$. The goal of the article is to drop the requirement of neatness and obtain the following theorem:
\begin{Th}\label{adacupcakeprincess}
Let $\Ob = \Gamma \bs \HH^{2n+1}$ be a finite volume hyperbolic  orbifold with $\Gamma \subset \SO_0(1,2n+1)$. For $m \in \N$, let $\tau(m)$ be a finite-dimensional irreducible representation of $\SO_0(1,2n+1)$ from Definition \ref{mimirep}
and $\tau'(m)$ be the restriction of $\tau(m)$ to $\Gamma$.
Let $E_{\tau(m)}\to \Ob$ be the associated flat vector orbibundle, and denote by $T_\Ob (\tau(m))$
 its analytic torsion as in Definition \ref{defineanal}. Then there exists $C(n)$ that depends only on $n$ such that  
 \begin{equation}\label{asexp}
\log T_\Ob (\tau(m)) = C(n)\cdot \vol(\Ob)\cdot m \cdot \dim(\tau(m))  + O(m^{\frac{(n+1)n}{2}} \log(m)) 
 \end{equation}
 as $\quad m \to \infty$.
\end{Th}
Unfortunately, so far we can study only the growth of the analytic torsion $T_\Ob(\tau(m))$, because there is no Cheeger-M\"uller theorem for orbifolds. However, there are partial results avaliable \cite{AlbinRochonSher,Lipnowski,Vertman}.

Let us now describe the proof of Theorem \ref{adacupcakeprincess}.
The first problem to define the analytic torsion. Let $\Delta_p(\tau(m))$ be the Hodge-Laplacian on $\Lambda^p(\Ob, E_{\tau(m)})$ as in Section \ref{Beweis}. Because the heat operator $e^{-t \Delta_p(\tau(m))}$ is not of trace class, we cannot define the
analytic torsion via  the usual zeta function regularization. But we can define a regularized trace $\Tr_{reg} e^{-t \Delta_p(\tau(m))}$ as in \cite{Park} or \cite{MP}.
Moreover, it turns out that the regularized trace equals the spectral 
side of the Selberg trace formula applied to the heat operator 
$e^{-t\Delta_p(\tau(m))}$. This implies that the asymptotic expansions $\Tr_{reg} e^{-t \Delta_p(\tau(m))}$ as 
$t\to +0$ and as $t\to\infty$. The existence of the asymptotics expansion allows us to define the spectral zeta
function $\zeta_{p}(s;\tau)$ as in the case of compact manifolds via the Mellin transform of the regularized trace. Moreover, as in the compact case the zeta function $\zeta_p(s;\tau(m))$ is regular at $s=0$, and we can 
define the 
analytic torsion $T_\Ob(\tau(m))\in\R^+$ with respect to  $E_{\tau}$ by
\begin{equation}\label{def-tor3}
T_\Ob(\tau(m)):=\exp{\left( \left.\frac{1}{2}\sum_{p=1}^{d}(-1)^pp\frac{d}{ds}
\zeta_{p}(s;\tau(m))\right|_{s=0}\right)}.
\end{equation}
We assume
that the highest weight $\tau_{n+1}$ of $\tau(m)$ satisfies $\tau_{n+1}\neq0$. Let
\begin{align*}
K(t,\tau(m)):=\sum_{p=0}^{2n+1}(-1)^p \,
p \, \Tr_{reg}(e^{-t\Delta_p(\tau(m))}).
\end{align*}
As the spectral zeta function $\zeta_p(s;\tau(m))$ is expressed via the Mellin transform of the heat kernel $K(t,\tau(m))$, we need to 
compute the Mellin transform of 
$K(t,\tau(m))$ at 0 to study the analytic torsion. For this we use the invariant Selberg trace formula  \cite{hoff} to express  $K(t,\tau)$ 
as:
\begin{equation}\label{trace7}
\begin{gathered}
K(t,\tau(m))=I(t;\tau(m))+H(t;\tau(m))+T(t;\tau(m))+\\ \mathcal{I}(t;\tau(m))+J(t;\tau(m))+
E(t;\tau(m))+\mathcal{E}^{cusp}(t;\tau(m))+\mathcal{J}^{cusp}(t;\tau(m)),
\end{gathered}
\end{equation}
where $I(t;\tau(m))$, $H(t;\tau(m))$ and $E(t;\tau(m))$ are the contributions of  identity, hyperbolic and elliptic conjugacy classes of $\Gamma$, respectively; $T(t;\tau(m))$, $\mathcal{I}(t;\tau(m))$ and $J(t;\tau(m))$ are
tempered
distributions which
are constructed out of the parabolic conjugacy classes of $\Gamma$; $\mathcal{E}^{cusp}(t;\tau(m))$ and $\mathcal{J}^{cusp}(t;\tau(m))$ are tempered distributions appearing due to the presence of non-unipotent stabilizers of the cusps of $\Ob$. 
Now we evaluate the Mellin transform of each term separately. It turns out that the leading term of the asymptotic expansion (\ref{asexp}) 
comes from $MI(\tau(m))$, which is a Mellin transform of $I(t;\tau(m))$ evaluated at zero.
It was proved in \cite{MP} that the contribution of $H(t;\tau(m))+T(t;\tau(m))+\mathcal{I}(t;\tau(m))+J(t;\tau(m))$ to the analytic torsion $T_\Ob(\tau(m))$ is of order $O(m^{\frac{(n+1)n}{2}} \log(m))$.
The contribution of elliptic elements to $T_\Ob(\tau(m))$ was studied in \cite{Fe2} and does not affect the leading term of (\ref{asexp}) as well. We are left with studying $\mathcal{J}^{cusp}(t;\tau(m))$ and $\mathcal{E}^{cusp}(t;\tau(m))$. The former distribution can be treated in a similar way as $J^{cusp}(t;\tau(m))$. The latter distribution is invariant and its Fourier transform was computed explicitly by Hoffmann \cite{HO}, which allows us to study its Mellin transform. 

The paper is organized as follows. In Section \ref{sectionprel} we provide the main definitions and notations. In Sections \ref{Eis} and \ref{sectionaboutSTF} we recall the theory of Eisenstein series and state the invariant Selberg trace formula, respectively. In Section \ref{woint} we express the weighted orbital integrals in a convenient form and prove the asymptotic expansion of the regularized heat trace. We define the analytic torsion and prove Theorem~\ref{adacupcakeprincess} in Section \ref{Beweis}.

\subsection{Acknowledgments} This paper is a part of the author's thesis, therefore she is grateful to her supervisor Werner M\"uller. I would like to heartly thank Werner Hoffmann for the kind and patient explanations about the Fourier transform of weighted orbital integrals.

\section{Preliminaries}
\label{sectionprel}

The purpose of this section is to introduce main notation and definitions, most importantly we define certain rays of representations of a discrete group acting on the hyperbolic space and give a classification of elements of $\Gamma \subset SO_0(1,2n+1)$ that appear in the Selberg trace formula. 
\subsection{Lie groups.}\label{subsection1}\label{whatareyoudoingman}
Let $G = \SO_0(1,2n+1)$, $K=\SO(2n+1)$. Denote $G = NAK$ be the standard Iwasawa decomposition of $G$, hence for each $g \in G$ there are uniquely determined elements $n(g) \in N$, $a(g) \in A$, $\kappa(g) \in K$ such that
$$ g = n(g) a(g) \kappa(g).$$
Let $M$ be the centralizer of $A$ in $K$, then 
$$M = \SO(2n).$$ 
Denote the Lie algebras of $G$, $K$, $A$, $M$ and $N$ by 
$\mathfrak{g}$,
$\mathfrak{k}$,
$\mathfrak{a}$,
$\mathfrak{m}$
and
$\mathfrak{n}$, respectively. Define the standard Cartan involution $\theta:\mathfrak{g} \to \mathfrak{g}$ by
$$\theta(Y) = - Y^t, \quad Y \in \mathfrak{g}.$$ Let $H: G \to \mathfrak{a}$ be defined by
\begin{equation}
\label{H-definition}
H(g) := \log a(g).
\end{equation}
Equipped with a certain invariant metric, $G/K$ is isometric to the hyperbolic space $\HH^{2n+1}$; see e.g. \cite[p.~6]{MP111}.

\subsection{Hyperbolic orbifolds.}
\label{hyporbsubsec}
 Consider a discrete subgroup $\Gamma \subset G$, $G=\SO_0(1,2n+1)$ such that $\Ob = \Gamma \bs \HH^{2n+1}$ is of finite volume. Recall the classification of the elements in $\Gamma$.
\begin{definition}
An element $\gamma \in \Gamma$ is called hyperbolic if 
$$ l(\gamma) := \inf_{x \in \HH^{2n+1}} d(x, \gamma x) > 0,$$
where $d(x,y)$ denotes the hyperbolic distance between $x$ and $y$.
\end{definition}
\begin{remark}
Some authors use the term "loxodromic" instead of hyperbolic.
\end{remark}
\begin{definition}
An element $\gamma \in \Gamma$ is called  elliptic if it is of finite order.
\end{definition}
An alternative definition is the following: an element $\gamma$ is elliptic if and only if it is conjugate to an element in $K$, so without loss of generality we may assume~$\gamma$ is of the form: 
\begin{equation}\label{pelmeni}
\gamma = \diag(\overbrace{\left( \begin{smallmatrix}
1 & 0  \\
0 & 1  \end{smallmatrix} \right), \ldots, \left( \begin{smallmatrix}
1 & 0  \\
0 & 1  \end{smallmatrix} \right)}^d, \overbrace{R_{\phi_{d+1}}, \ldots, R_{\phi_{n+1}}}^{n-d+1}),
\end{equation}
 where $R_\phi = \left(      \begin{smallmatrix} \cos \phi & \sin \phi \\ -\sin \phi & \cos \phi \end{smallmatrix}  \right)$.
\begin{definition}
An elliptic element $\gamma$ is regular if the stabilizer $G_\gamma$ of $\gamma$ in $\Gamma$ equals $\SO_0(1,1) \times \SO(2)^{n-1}$.
\end{definition} 
\begin{remark}
The calculation of the ordinary and weighted orbital integrals corresponding to elliptic elements appearing in the right hand side of the Selberg trace formula depends on whether an elliptic element is regular; see \cite{Fe1} and Section \ref{woint}.
\end{remark}
Let $\mathfrak{P}$ be a fixed
set of representatives of $\Gamma$-nonequivalent proper cuspidal parabolic 
subgroups of~$G$. If $\Gamma \bs \HH^{2n+1}$ is of finite volume, then the number of cusps $\kappa := \#\mathfrak{P}$ is finite. Without loss 
of generality we can assume that $P_{0}:=MAN\in\mathfrak{P}$. 
For every $P\in\mathfrak{P}$, there exists $k_{P}\in K$ such that
\[
P=N_{P}A_{P}M_{P}
\]
with $N_{P}=k_{P}Nk_{P}^{-1}$, $A_{P}=k_{P}Ak_{P}^{-1}$, and
$M_{P}=k_{P}Mk_{P}^{-1}$. 
Note that the structure of a cusp corresponding to a cuspidal parabolic subgroup $P \in \mathfrak{P}$ depends on $\Gamma \cap P$.  (to Dmitry: I am reviewing well known material only in this paragraph)
\begin{definition}
\label{neatness}
The group $\Gamma$ is neat if $\Gamma \cap P = \Gamma \cap N_P$.
\end{definition}
If $\Gamma$ is neat, then the cross-section of a cusp $P$ is a torus. The known results about the analytic torsion of $\Gamma \bs \HH^{2n+1}$ require the group $\Gamma$ to be neat \cite{MP2, Park}. As this excludes various important arithmetic groups, we wish to allow $\Gamma$ to be not neat. For example, we allow $\Gamma$ to have elements of the following type:
\begin{definition}
Let $\gamma \in \Gamma$ be an elliptic element. If there exists $P \in \mathfrak{P}$ such that $\gamma \in \Gamma \cap P$, then $\gamma$ is called a cuspidal elliptic element.
\end{definition}
\begin{exmp}
Let $G = \textnormal{SL}_2(\C)$, then the group $\Gamma = \textnormal{PSL}_2(\Z\oplus (-1+i \sqrt{3})\Z / 2)$ is not neat, as $\left( \begin{smallmatrix}
(-1+i \sqrt{3})/2 & 0  \\
0 & (1+i \sqrt{3})/2  \end{smallmatrix} \right)$ is a cuspidal elliptic element. The cross-section of the only cusp is an orbifold with 3 singular points of order 3.
\end{exmp}
Recall that a Levi component $L$ is a centralizer of $A$ in $G$, thus $L = MA$. In order to formulate the Selberg trace formula, we need to introduce the following set of elements of $\Gamma$:
\begin{definition}\label{15:21}
Denote by $\Gamma_M(P)$ the set of projections to  $L$  of $\Gamma \cap P$.
\end{definition}
\begin{remark}
\label{marinad}
By \cite[p. 5]{War}, $\Gamma \cap P \subset MN$, hence $\Gamma_M(P) \subset M$. This implies that the set $\Gamma_M(P)$ is finite and each its element is of finite order.
\end{remark}
\begin{remark}
The set of $\Gamma$-conjugacy classes of cuspidal elliptic elements of $\Gamma$ does not necessarily coincide with $\Gamma_M(P)$. For example, take $G = \textnormal{SL}(2,\C)$, $\Gamma=\textnormal{PSL}(2,\Z[i])$. Then $P = \left( \begin{smallmatrix}
\cdot & \cdot \\ 0 & \cdot
\end{smallmatrix}\right)$, $MA=\left(\begin{smallmatrix}
e^{i \phi + r} & 0 \\ 0 & e^{-i \phi - r}
\end{smallmatrix} \right)$, where $r >0$ and $\phi\in[0,2 \pi)$, and $\Gamma \cap P = \left(\begin{smallmatrix}
i & \cdot \\ 0 & -i
\end{smallmatrix} \right)$. Hence $\Gamma_M(P) = \left(\begin{smallmatrix}
i & 0 \\ 0 & -i
\end{smallmatrix} \right)$.

On the other hand, $\gamma \in \left(\begin{smallmatrix}
i & 1 \\ 0 & -i
\end{smallmatrix} \right) \in \Gamma$ is an elliptic element that stabilizes the cusp, but it is not $\Gamma$-conjugated to $\left(\begin{smallmatrix}
i & 0 \\ 0 & -i
\end{smallmatrix} \right)$. To prove this, let 
$\bsm a & b \\ c & d \esm \in \textnormal{SL}(2, \Z[i])$ 
such that 
$\bsm a & b \\ c & d \esm^{-1} \cdot \bsm i & 1 \\ 0 & -i \esm \cdot \bsm a & b \\ c & d \esm = \bsm i & 0 \\ 0 & -i \esm$. The last equality implies that $c=0$, $d = -2 i b$, hence $ a = -1 / (2id)$. It follows that $a$ and $d$ cannot belong to $\Z[i]$ simultaneously.
\end{remark}
 To finish this subsection we recall the following lemma:
 \begin{Lemma}[Selberg lemma]\label{umschreiben}
 A cofinite group $\Gamma \subset \SO_0(1,2n+1)$ has a normal torsion free subgroup $\Gamma'$ of finite index.
 \end{Lemma}

\subsection{Lie algebras}\label{liealg}
Denote by $E_{i,j}$ the matrix in $\mathfrak{g}$ whose $(i,j)$'th entry is 1 and the other entries are~0. Let
\begin{equation*}
\begin{gathered}
H_1 := E_{1,2} + E_{2,1},\\
H_j := i (E_{2j-1,2j} - E_{2j,2j-1}), \quad j = 2, \ldots, n+1.
\end{gathered}
\end{equation*}
Then $\mathfrak{a} = \R H_1$ and let  $\mathfrak{b} = i \R H_2 + \ldots + i \R H_{n+1}$ be the standard Cartan subalgebra of $\mathfrak{m}$. Moreover, $\mathfrak{h} = \mathfrak{a} \oplus \mathfrak{b}$ is a Cartan subalgebra of $\mathfrak{g}$. Define $e_i \in \mathfrak{h}_{\C}^*$ with $i = 1, \ldots, n+1$, by
\begin{equation}\label{makaronina}
 e_i(H_j) = \delta_{i,j}, \, 1\le i,j\le n+1.
\end{equation}
The sets of roots of $(\mathfrak{g}_\C, \mathfrak{h}_{\C})$ and   $(\mathfrak{m}_{\C}, \mathfrak{b}_{\C})$ are given by
\begin{equation}\label{nichegouzhenesvyazano}
\begin{gathered}
\Delta(\mathfrak{g}_\C, \mathfrak{h}_{\C}) = \{ \pm e_i \pm e_j, 1 \le i < j \le n+1\},\\
\Delta(\mathfrak{m}_{\C}, \mathfrak{b}_{\C}) = \{  \pm e_i \pm e_j, 2\le i < j \le n+1   \}.
\end{gathered}
\end{equation}
We fix the positive systems of roots by 
\begin{equation}\label{mojbambino}
\begin{gathered}
\Delta^+(\mathfrak{g}_\C, \mathfrak{h}_{\C}) = \{  e_i \pm e_j, 1 \le i < j \le n+1\},\\
\Delta^+(\mathfrak{m}_{\C}, \mathfrak{b}_{\C}) = \{   e_i \pm e_j, 2\le i < j \le n+1   \}.
\end{gathered}
\end{equation}
The half-sum of the positive roots $\Delta^+(\mathfrak{m}_{\C}, \mathfrak{b}_{\C})$ equals 
\begin{equation}\label{halfsummy}
\rho_M = \sum_{j=2}^{n+1} \rho_j e_j,\quad \rho_j = n+1-j.
\end{equation}
Let $M'$ be the normalizer of $A$ in $K$ and let $W(A)=M'/M$ be the restricted Weyl group. It has order~2 and acts on finite-dimensional representations of $M$ \cite[p.~18]{Pf}. Denote by $w_0$ the non-identity element of $W(A)$.

\subsection{Principal series parametrization.}\label{princyyyy} Let $\sigma: M \mapsto \End(V_\sigma)$ be a finite-dimensional irreducible representation of $M$. 
\begin{definition}
We define $\mathcal{H}^\sigma$ to be the space of measurable functions $f:K\mapsto V_\sigma$ such that
\begin{enumerate}
\item $f(mk)=\sigma(m) f(k)$ for all $k \in K$ and $m \in M$;
\item $\int_K || f(k)||^2 dk < \infty$.
\end{enumerate}
\end{definition}
Recall $H: G \to \mathfrak{a}$, $\kappa: G \to K$ are as in Subsection \ref{subsection1} and $e_1 \in \mathfrak{h}_\C^*$ is as in Subsection \ref{liealg}. For $\lambda \in \R$ define the representation $\pi_{\sigma, \lambda}$ of $G$ on $\mathcal{H}^\sigma$ by the following formula:
$$ \pi_{\sigma,\lambda} (g) f(k) := e^{(i \lambda e_1 + \rho) (H(kg))} f(\kappa(kg)),$$
where $f\in \mathcal{H}^\sigma$, $g\in G$.
\subsection{Representations.} Fix $\tau_1, \ldots, \tau_{n+1} \in \N,$ such that $\tau_1 \ge \tau_2 \ge \ldots \ge \tau_{n+1}$. Recall that $n~=~\frac{\dim(\Ob)-1}{2}$.
\begin{definition}\label{mimirep}
For $m \in \N$ denote by $\tau(m)$ the finite-dimensional representation of $G$ with highest weight 
$$(m + \tau_1)e_1 + \ldots + \ldots (m+\tau_{n+1}) e_{n+1}.$$
\end{definition}
This is a ray of representations which will be the focus of the article.
\begin{definition}\label{dertttt} Let $\tau$ be the finite-dimensional irreducible representation of $G$ with highest weight $\tau_1 e_1 + \ldots + \tau_{n+1} e_{n+1}$.
The denote by $\sigma_{\tau,k}$ be the representation of $M$ with  highest weight
$$\Lambda_{\sigma_{\tau,k}}:= (\tau_2 +1)e_2 + \ldots + (\tau_k +1)e_{k+1}+\tau_{k+2}e_{k+2} + \ldots + \tau_{n+1}e_{n+1}.$$
\end{definition}

\subsection{Differential operators}
\label{lightmycandle}
Let $\nu$ be a finite-dimensional unitary representation $\nu$
of $K$ over $(V_\nu, \langle \cdot, \cdot \rangle_\nu)$. Let $\widetilde{E}_\nu := G \times_{\nu} V_\nu$ be the associated 
homogeneous vector bundle over $\HH^{2n+1}$ and $E_\nu:= \Gamma \bs \widetilde{E}_\nu$ be the corresponding locally homogeneous orbibundle over $\Ob$. The smooth sections of $\widetilde{E}_\nu$ can be identified with
\begin{equation*}
\begin{gathered}
C^\infty(G,\nu):= \left\lbrace  f\in C^\infty(G, V_\nu), \quad  f(gk)=\nu(k^{-1}) f(g) \quad \forall g\in G , \right. \\ \left. \forall k \in K\right\rbrace.
\end{gathered}
\end{equation*}
It follows that the smooth sections of $E_\nu$ can be identified with
$$C^\infty(\Gamma \bs G,\nu):= \left\lbrace  f\in C^\infty(G,\nu), \quad  f(\gamma g) = f(g) \quad \forall g\in G, \gamma \in \Gamma \right\rbrace.$$
The spaces $L^2( G, \nu)$ and $L^2(\Gamma \bs G, \nu)$ are defined as usual.
\begin{definition}
\label{yanenavizhu}
Let $\widetilde{A}_\nu$ be the differential operator that acts on $C^\infty(G,\nu)$ by $-R(\Omega)$, where $R$ is the right regular representation of $G$ and $\Omega$ is the Casimir element of $G$. Let $A_\nu$ be its push-forward to $C^\infty(\Gamma \bs G, \nu)$.
\end{definition}
\begin{Proposition}\cite[Proposition 1.1]{Mi}
$\widetilde{A}_\nu$ and $A_\nu$ are essentially self-adjoint and bounded from below.\qed
\end{Proposition}

Let $e^{-t A_\nu}$ be the semigroup of $A_\nu$ on $L^2(\Gamma \bs G, \nu)$, let $H_t^\nu(g)$ be its convolution kernel and
\begin{equation}
\label{htdef}
h_t^\nu(g):= \tr H_t^\nu(g), \quad g \in G,
\end{equation}
where $\tr$ denotes the trace in $\End( V_\nu)$.

\subsection{Truncation.} Recall that to define the regularized trace $\Tr_{reg} e^{-t A_\nu}$ \cite{MP,Par} we need to introduce the height function on every cusp. To do so, for each $P\in\mathfrak{P}$ define 
\[
\iota_P\colon \R^+\to A_P
\]
by $\iota_P(t):=a_P(\log(t))$. For $Y>0$, let 
\[
A^{0}_{P}\left[Y\right]:=(\iota_P(Y),\iota(\infty)).
\]
Let $\Gamma' \subset \Gamma$ be as in  Lemma \ref{umschreiben}.
Then there exists a $Y_{0}>0$ and for every $Y\geq Y_{0}$ a compact connected
subset $C(Y)$ of $G$  such that in the sense of a disjoint union one has
\begin{align}\label{Zerlegung des FB}
G=\Gamma' \cdot C(Y)\sqcup\bigsqcup_{P\in\mathfrak{P}}\Gamma' \cdot
N_{P}A^{0}_{P}\left[Y\right]K
\end{align}
and such that 
\begin{align}\label{Eigenschaft des FB}
\gamma\cdot N_{P}A^0_{P}\left[Y\right]K\cap N_{P}A_{P}^{0}\left[Y\right]K\neq
\emptyset\Leftrightarrow\gamma\in \Gamma'_{N}. 
\end{align}
\begin{definition}\label{makeupyourmind}
For $P\in\mathfrak{P}$ let $\chi_{P,Y}$ be the characteristic function of
$N_PA_{P}^0\left[Y\right]K\subset G$.
\end{definition}
\begin{remark}
The truncation of a manifold $\Gamma' \bs \HH^{2n+1}$ induces a trunction of an orbifold $\Gamma \bs \HH^{2n+1}$. 
\end{remark}

\section{Eisenstein series}
\label{Eis}
In this section we recall the definition and main properties of  Eisenstein series \cite{War}. Recall that $\Gamma$ is a cofinite lattice in $SO_0(1,d)$ with $d=2n+1$, $n \in \N$.
\begin{definition}
For $P=M_PA_PN_P\in\mathfrak{P}$ as in Subsection \ref{hyporbsubsec}, let  $\mathcal{E}_{P}$ be the space of measurable functions $\Phi: G \to \C$ such that
\begin{enumerate}
\item $\Phi(g \gamma) = \Phi(\gamma)$ for all $g \in (\Gamma \cap P)N_P A_P$,
\item $\Phi|_K$ is square-integrable.
\end{enumerate}
\end{definition}
\begin{definition}
For $\Phi\in\mathcal{E}_P$ and $\lambda\in\C$, put
\[
\Phi_\lambda(x):=e^{(\lambda +(d-1)/2)H(x)}\Phi(x),
\]
where $H(x)$ is as in (\ref{H-definition}). Then for each $\lambda\in\C$ there is a representation $\pi_{P,\lambda}$ of
$G$
on $\mathcal{E}_{P}$ defined by
\[
(\pi_{P,\lambda}(y)\Phi)_\lambda(x):=\Phi_\lambda(xy).
\] 
The representation  $\pi_{P,\lambda}$ is unitary for $\lambda\in i\R$.
\end{definition}
\begin{definition}\label{sweetsweet}
We define an inner product $\langle \cdot, \cdot \rangle: \mathcal{E}_P \times \mathcal{E}_P \to \C$ by
\begin{equation}\label{kuscheln}
\langle \Phi, \Psi \rangle :=  \int_K \int_{M / \Gamma_M} \Phi(km) \bar{\Psi}(km) \,dk\, dm,
\end{equation} 
where $\Gamma_M = \Gamma \cap M \cdot N/\Gamma \cap N$. 
\end{definition}
\begin{definition}
Denote by $\mathcal{E}^0_P$ the
subspace of $\mathcal{E}_P$ consisting of all right $K$-finite and left
$\mathfrak{Z}_M$-finite functions, where $\mathfrak{Z}_M$ denotes the center of the universal enveloping algebra of $\mathfrak{m}_\C$.
\end{definition}
\begin{definition}\label{enttaeuscht}
For $\Phi\in\mathcal{E}_P^0$, the Eisenstein series $E(P,\Phi,\lambda,x)$  is defined 
by
\[
E(P,\Phi,\lambda,x):=\sum_{\gamma\in\Gamma\cap P\bs \Gamma}\Phi_\lambda(\gamma
x).
\]
\end{definition}
It converges absolutely and uniformly on compact subsets of 
$\{\lambda\in\C\colon \RRe(\lambda)>(d-1)/2\}\times G$, and has a 
meromorphic
extension to $\C$. For $P' \in\mathfrak{P}$, the constant term 
$E_{P^\prime}(P,\Phi,\lambda)$ of $E(P,\Phi,\lambda)$  is
 defined by
\begin{equation}
E_{P^\prime}(P,\Phi,\lambda,x):=\frac{1}{\vol(\Gamma\cap P'\bs
N_{P^\prime})}
\int_{\Gamma\cap P'\bs N_{P^\prime}}E(P,\Phi,\lambda,y 
x)\;dy.
\end{equation}
Note that the function $\Phi$ is left $(\Gamma \cap P)$-invariant, hence
\begin{equation}\label{c-term1}
E_{P^\prime}(P,\Phi,\lambda,x) = \frac{1}{\vol(\Gamma\cap N_{P^\prime}\bs
N_{P^\prime})}
\int_{\Gamma\cap N_{P^\prime}\bs N_{P^\prime}}E(P,\Phi,\lambda,n^\prime
x)\;dn^\prime.
\end{equation}
Moreover, there exist linear maps \cite[(3.9)]{MP}
$$ c_{P'|P}(w:\lambda): \mathcal{E}_P \to \mathcal{E}_{P'},$$
which are meromorphic functions of $\lambda \in \C$, such that 
\begin{equation}\label{c-term2}
E_{P^\prime}(P,\Phi,\lambda,x)=\sum_{w\in W(A_P,A_{P^\prime})}
e^{(w\lambda+(d-1)/2)(H_{P^\prime}(x))}
\left(c_{P^\prime|P}(w\colon\lambda)\Phi\right)(x)
\end{equation}
for $w \in W(A_P, A_{P'})$ as in \cite[(3.9)]{MP}.
\begin{definition}\label{sweetsacrifice} Put
\[
\boldsymbol{\mathcal{E}}:=\bigoplus_{P\in\mathfrak{P}}\mathcal{E}_P,\quad \boldsymbol{\mathcal{E}}^0:=\bigoplus_{P\in\mathfrak{P}}\mathcal{E}^0_P, \quad \pi_\lambda=\bigoplus_{P\in\mathfrak{P}} \pi_{P,\lambda}.
\] We define the inner product $\langle \cdot , \cdot \rangle$ on $\boldsymbol{\mathcal{E}}$ using Definition \ref{sweetsweet}. 
\end{definition}
\begin{definition}\label{fallingforever}
For $\boldsymbol{\Phi} \in \boldsymbol{\mathcal{E}}$, define
\[
{E}(\mathbf{\Phi},\lambda,x):=\sum_{P\in\mathfrak{P}}E(P,\Phi_P,\lambda,x), 
\quad {E_{P'}}(\mathbf{\Phi},\lambda,x):=\sum_{P\in\mathfrak{P}}E_{P'}(P,\Phi_P,\lambda,x).
\]
\end{definition}
\begin{definition}
Let $w_0$ be the nontrivial element of $W(A)$. The operators 
$c_{P^\prime|P}(k_P'w_0k_P^{-1}:\lambda)$ can be combined into a linear operator
\[
\mathbf{C}(\lambda)\colon \boldsymbol{\mathcal{E}}^0\to \boldsymbol{\mathcal{E}}^0,
\]
which is a meromorphic function of $\lambda \in \C$. 
\end{definition}

\begin{definition}
We define the truncated Eisenstein series by:
$$E^Y(\Phi,\lambda,x):=E(\Phi,\lambda,x)
-\sum_{P\in\mathfrak{P}} \sum_{\gamma\in\Gamma\cap
P\backslash\Gamma}\chi_{P,Y}
(\gamma g)E_{P}(\Phi,\lambda,\gamma g),$$
where $E(\Phi, \lambda, x)$ and $E_{P}(\Phi,\lambda,\gamma g)$ are from Definition \ref{fallingforever} and  $\chi_{P,Y}(\gamma g)$ is from Definition \ref{makeupyourmind}.
\end{definition}

\begin{Lemma}[Maass-Selberg relations]\label{maassive_donut}
Let  $\Phi,\Psi\in\boldsymbol{\mathcal{E}}^0$ and $\lambda\in\mathfrak{a}^*$. The lemma below follows from the proof of \cite[Lemma 4.3]{Pf} with minor changes. Note that the inner product $\langle \cdot , \cdot \rangle$ in this lemma should be understood as in Definition \ref{sweetsacrifice}.
\begin{equation*}
\begin{gathered}
\int_{\Gamma\bs G}E^Y(\Phi,i\lambda,x)\overline{E^Y}(\Psi,i\lambda,x)\;dx
=-\left<\mathbf{C}(-i\lambda)\frac{d}{dz}\mathbf{C}(i\lambda)\Phi,\Psi\right>+\\
2\left<\Phi,\Psi\right>\log{Y}+\frac{Y^{2i\lambda}}{2i\lambda}
\left<\Phi,\mathbf{C}(i\lambda)\Psi\right>-\frac{Y^{-2i\lambda}}{2i\lambda}
\left<\mathbf{C}(i\lambda)\Phi,\Psi\right>.
\end{gathered}
\end{equation*}
\end{Lemma}

\section{Trace formula}\label{sectionaboutSTF}
In this section we recall the invariant Selberg trace formula, define the regularized trace $\textnormal{Tr}_{reg}\left(e^{-tA_\nu}\right)$ and express it as the spectral side of the Selberg trace formula. The main theorem of this section will be Theorem~\ref{whattheyknow}.

First we recall some facts from \cite[Sections 1-3]{War}. Let $\pi_\Gamma$ be the right-regular representation of $G$ on $L^2(\Gamma \bs G)$. Then there exists an orthogonal decomposition
\begin{equation}
\label{decup}
L^2(\Gamma \bs G) = L^2_d(\Gamma \bs G) \oplus L^2_c(\Gamma \bs G)
\end{equation}
into closed $\pi_\Gamma$-invariant subspaces. The restriction $\pi_\Gamma$ to $L^2_c(\Gamma \bs G)$ is isomorphic to the direct integral over all unitary principle series representations of $\Gamma$. The restriction of $\pi_\Gamma$ to $L_d^2(\Gamma \bs G)$ decomposes into the orthogonal direct sum of irreducible unitary representations of $\Gamma$.
\begin{definition}
Let $\al \in C^\infty(G)$ be a $K$-finite Schwarz function. Denote by $\pi_\Gamma(\al)$ the following operator on $L^2(\Gamma \bs G)$:
\begin{equation}
\pi_\Gamma(\al) f(x) := \int_G \al(g) f(xg) dg.
\end{equation}
\end{definition}
Note that relative to (\ref{decup}) one has a splitting:
\begin{equation*}
\pi_\Gamma(\al)=\pi_{\Gamma,d}(\al) \oplus \pi_{\Gamma,c}(\al).
\end{equation*}
The operator $\pi_{\Gamma,d}$ is of trace class by an extension of \cite[Theorem~I.1]{Donnelly}. We recall our main tool, namely a special case of the invariant trace formula stated in \cite[Theorem 6.4]{hoff}:
\begin{Th} 
 For a $K$-finite Schwarz function~$\al\in C^\infty(G)$ we have
 \begin{equation}
 \begin{gathered}
 \Tr\left(\pi_{\Gamma,d}(\al)\right)=I(\al)+H(\al
  )+T(\al)+\mathcal{I}(\al)+R(\al)+\\ \mathcal{S}(\al)+E(\al)+E^{cusp}(\al) + \mathcal{J}^{cusp}(\al).
 \end{gathered}
 \end{equation}
 Above $I(\al)$, $H(\al)$, $T(\al)$, $\mathcal{I}(\al)$, $R(\al)$, and $\mathcal{S}(\al)$ and  are as in \cite[(6.1)]{Pf},  \cite[(6.3)]{Pf}, \cite[(6.9)]{Pf},  \cite[(6.12)]{Pf}, \cite[(6.17)]{Pf}, and \cite[(6.13)]{Pf}, respectively. The distributions $E(\al)$, $E^{cusp}(\al)$ and $\mathcal{S}^{cusp}(\al)$ shall be defined later in the section. 
 \end{Th}
\begin{definition}
Denote by $\{ \gamma  \}$ the conjugacy class of $\gamma \in \Gamma$,  then 
$$ E(\al) = \sum_{ \{\gamma\} \textnormal{ elliptic}} \vol(\Gamma_\gamma \bs G_\gamma) \int_{G_\gamma \bs G} \al(x \gamma x^{-1}) d x,$$
where $\Gamma_\gamma$ and $G_\gamma$ denote the centralizers of $\gamma$ in $\Gamma$ and $G$, respectively.
\end{definition}
\begin{definition}\label{tee}
Let $\gamma \in \Gamma_M(P)$, where $\Gamma_M(P)$ is from Definition \ref{15:21}, and let $\al \in C^\infty(G)$ be a $K$-finite Schwarz function. We define the weighted orbital integral $ J_L (\gamma, \al)$ by
$$ J_L (\gamma, \al) := |D_G(\gamma)|^{1/2} \int_{G / G_\gamma} \al(x a x^{-1}) v(x) dx,$$
where $D_G(\gamma)$ and $v(x)$ are as in \cite[p. 55]{HO}.
\end{definition}
\begin{Proposition}\label{orbringherbacktome} \cite[p.~58]{HO} The weighted integral in Definition~\ref{tee} is 
not an invariant distribution, but the distribution $I_L(\gamma, \al)$ below is invariant:
\begin{equation}
\label{aqwrer}
\begin{gathered}
I_L(\gamma, \al) := J_L(\gamma, \al) - \\ \frac{1}{2 \pi i} \sum_{\sigma \in \widehat{M}} \int_{D_{\epsilon}} \Theta_{\breve{\sigma}_{-\lambda}}(\gamma) \cdot \Tr\left(J_{\bar{P}_{0}|P_{0}}(\sigma,z)^{-1}\frac{d}{
dz}J_{\bar{P}_{0}|P_{0}}(\sigma,z)\pi_{\sigma,z}(\alpha
)\right)dz,
\end{gathered}
\end{equation}
where $J_{\bar{P}_{0}|P_{0}}$ and $\breve{\sigma}_\lambda$ are defined as in \cite[(6.6)]{MP}; $D_{\epsilon}$ is the path which is the union of
$\left(-\infty,-\epsilon\right]$, $H_{\epsilon}$ and
$\left[\epsilon,\infty\right)$, where $H_{\epsilon}$ is the half-circle from $-\epsilon$ to $\epsilon$ in the lower half-plane oriented counter-clockwise; $\pi_{\sigma, z}$ is defined in Subsection \ref{princyyyy}; $\widehat{M}$ is the set of equivalence classes of irreducible unitary representations of $M$; and 
$$\Theta_{\breve{\sigma}_{-\lambda}}(m_\gamma a_\gamma) = e^{- i \lambda t} \cdot \Theta_{\breve{\sigma}}(\sigma), \quad \Theta_{\breve{\sigma}}(m_\gamma) = \tr \; \breve{\sigma}(m_\gamma),$$
for $\gamma$ conjugated to $m_\gamma a_\gamma$, where $a_\gamma = \bsm e^t & 0 \\ 0 & e^{-t} \esm$ and $m_\gamma \in M$. 
\end{Proposition}

\begin{definition}
\label{ibringthemalltolight}
Let 
\begin{equation*}
\begin{gathered}
E^{cusp}(\al) := \sum_{ \gamma \in \Gamma_M(P)} C(\gamma)  \cdot I_L(\gamma, \al),\\
\mathcal{J}^{cusp}(\al) := \sum_{ \gamma \in \Gamma_M(P)} C(\gamma) \cdot (-I_L(\gamma, \al) + J_L(\gamma, \al)),
\end{gathered}
\end{equation*}
where $C(\gamma)$ is a constant from \cite{hoff}. 
\end{definition}
\begin{remark}
The information we need about $C(\gamma)$ is that it depends only on $\gamma$. 
\end{remark}
\begin{remark}
\label{apocalyptiss}
 Recall that $\Gamma_M(P) \subset M$, hence $\Theta_{\breve{\sigma}_{-\lambda}}(\gamma)$ does not depend on $\lambda$. Definition \ref{ibringthemalltolight} and (\ref{aqwrer}) imply
$$\mathcal{J}^{cusp}(\al) = \sum_{ \gamma \in \Gamma_M(P)} C'(\gamma) \cdot (-I_L(1, \al) + J_L(1, \al))$$
for some new constant $C'(\gamma)$.
\end{remark}

\subsection{Trace regularization}
Let $L^2(\Gamma \bs G, \nu)$ be as in Subsection \ref{lightmycandle}. Note that $L^2(\Gamma \bs G, \nu)$ decomposes with respect to (\ref{decup}) in the following way:
\begin{align*}
L^2(\Gamma\backslash
G,\nu)=L^2_d(\Gamma\backslash
G,\nu)\oplus L^2_c(\Gamma\backslash
G,\nu). 
\end{align*}
Denote by $A_\nu^d$ the restriction of $A_\nu$ from Definition \ref{yanenavizhu} to $L^2_d(\Gamma\backslash
G,\nu)$; its spectrum is discrete, and by \cite{MP} the operator $e^{-t A^d_\nu}$ is of trace class.
The main result of this subsection is the following theorem:
\begin{Th}\label{whattheyknow}
Let the regularized trace of $e^{-tA_\nu}$ be defined as:
\begin{equation}\label{regtrace2}
\begin{gathered}
\textnormal{Tr}_{reg}\left(e^{-tA_\nu}\right):=\Tr\left(e^{-tA^d_\nu}\right)+
\sum_{\substack{\sigma\in\hat{M};\sigma=w_0\sigma\\
\left[\nu:\sigma\right]\neq
0}}e^{tc(\sigma)}\frac{\Tr(\widetilde{\boldsymbol{C}}(\sigma,\nu,0))}{4}\\
-\frac{1}{4\pi}\sum_{\substack{\sigma\in\hat{M}\\
\left[\nu:\sigma\right]\neq
0}}\int_{\R}e^{-t\left(\lambda^2-c(\sigma)\right)}
\Tr\left(\widetilde{\boldsymbol{C}}(\sigma,\nu,-i\lambda)
\frac{d}{dz}\widetilde{\boldsymbol{C}}(\sigma,\nu,i\lambda)\right)\,d\lambda.
\end{gathered}
\end{equation}
Then the right hand side of (\ref{regtrace2}) equals the spectral side
of the Selberg trace formula applied to $\exp(-tA_\nu)$, and hence
$$\Tr_{reg}(e^{-tA_{\nu}})=I(h^{\nu}_{t})+H(h^{\nu}_{t})+T(
h^{\nu}_{t})+\mathcal{I}(h^{\nu}_{t})+J(h^{\nu}_{t})+E^{cusp}(h^{\nu}_{t}) + \mathcal{J}^{cusp}(h^{\nu}_{t}),$$
where $h^{\nu}_{t}$ is from (\ref{htdef}).
\end{Th}
For a proof, we need the following lemma:
\begin{Proposition}\cite[Theorem 4.7]{War}
The operator
$\pi_{\Gamma,c}( h_t^\nu)$ is an integral operator with kernel $h_c^\nu(t;x,y)$
given by
\begin{equation}\label{contkernel}
h^\nu_c(t;x,y)=\frac{1}{4\pi}\sum_{m,n\in I}\int_\R 
\langle\boldsymbol{\pi}_\lambda( h_t^\nu)e_m,e_n\rangle E(e_n,i\lambda,x)
\overline{E(e_m,i\lambda,y)}\;d\lambda,
\end{equation}
where $\{e_n\colon n\in I\}$ is an orthonormal basis of $\boldsymbol{\mathcal{E}}$ as in \cite[p. 40]{War}. 
Furthermore, the kernel
$
h_d^\nu(t;x,y)=h^\nu(t;x,y)-h_c^\nu(t;x,y)
$
is integrable along the diagonal and
\[
e^{-t A_d^\nu} = \Tr(\pi_{\Gamma,d}(\alpha))=\int_{\Gamma\bs G} h^d_\nu(t;x,x)\;dx.
\]
\end{Proposition}
By \cite[p. 82]{War}, 
\begin{align*}
\int_{\mathbb{R}}\int_{\Gamma\backslash
G}{{\left|\sum_{k,l}\left<\boldsymbol{\pi}_{i\lambda}(h_t^\nu)e_l,e_k\right>
E^Y(e_k,i\lambda,x)\overline{E}^Y(e_l,i\lambda,x)\right|dx}d\lambda}
<\infty.
\end{align*}
Using Lemma \ref{maassive_donut}, one obtains
\begin{equation*}
\begin{gathered}
\int_{X(Y)}h^\nu_c(t;x,x)\;dx=
\sum_{\substack{\sigma\in\hat{M}\\\sigma=w_0\sigma \\  \left[\nu:\sigma\right]\neq0}
}
\frac{\Tr\left(\boldsymbol{\pi}_{\sigma,0}(h_t^\nu)\mathbf{C}
(\sigma,\nu,0)\right)}{4}+\\ \sum_{\substack{\sigma\in\hat{M}\\
\left[\nu:\sigma\right]\neq
0}}\biggl(\frac{\kappa e^{tc(\sigma)}\log{Y}\dim(\sigma)}{\sqrt{4\pi t}} - \\
\frac{1}{4\pi}\int_{\R}\Tr\left(\boldsymbol{\pi}_{\sigma,
i\lambda}(h_t^\nu)\mathbf{C}(\sigma,\nu,-i\lambda)\frac{d}{dz}
\mathbf{C}(\sigma,\nu,i\lambda)\right)\,d\lambda
\biggr)+o(1),
\end{gathered}
\end{equation*}

as $Y\to\infty$, and hence
\begin{equation}\label{itsmyworld}
\begin{gathered}
\int_{X(Y)}h^\nu(t;x,x)\,dx=
\sum_{\substack{\sigma\in\hat{M}\\
\left[\nu:\sigma\right]\neq0}}
\frac{\kappa e^{tc(\sigma)}\dim(\sigma)\log{Y}}{\sqrt{4\pi t}}
+\sum_j e^{-t\lambda_j} \\
+\sum_{\substack{\sigma\in\hat{M}\\\sigma=w_0\sigma\\
\left[\nu:\sigma\right]\neq
0}}e^{tc(\sigma)}\frac{\Tr(\widetilde{\boldsymbol{C}}(\sigma,\nu,0))}{4} - \\
\frac{1}{4\pi}\sum_{\substack{\sigma\in\hat{M}\\
\left[\nu:\sigma\right]\neq
0}}\int_{\R}e^{-t\left(\lambda^2-c(\sigma)\right)}
\Tr\left(\widetilde{\boldsymbol{C}}
(\sigma,\nu,-i\lambda)\frac{
d}{dz}\widetilde{\boldsymbol{C}}(\sigma,\nu,i\lambda)\right)\,d\lambda
+o(1)
\end{gathered}
\end{equation}
It follows that $\int_{X(Y)}h^\nu(t;x,x)\,dx$ has an asymptotic expansion as $Y$ tends to $\infty$; it is easy to check that $\textnormal{Tr}_{reg}\left(e^{-tA_\nu}\right)$ from (\ref{regtrace2}) is the constant term in (\ref{itsmyworld}). Moreover, the right hand side of (\ref{itsmyworld}) equals the spectral side of Selberg trace formula from  \cite[Theorem 4.2]{hoff} applied to $e^{-t A_\nu}$.
\qed

\section{Fourier transform of the weighted orbital integrals}
\label{woint}
In this section we recall 
the Fourier transform of the distributions $\mathcal{I}(\alpha)$ and $E^{cusp}(\al)$ \cite{HO} and express it in a more convenient way.
\begin{Th}\label{FTorbint}\cite[Corollary~on~p.96]{HO}
For every $K$-finite $\alpha\in C^{2}(G)$ one has
\begin{align*}
\mathcal{I}(\alpha)=\frac{\kappa}{4\pi}\sum_{\sigma\in\hat{M}}\int_{\mathbb{R
}}{\Omega(\check{\sigma},-\lambda)\Theta_{\sigma,\lambda}(\alpha)d\lambda},
\end{align*}
where
\begin{equation*}
\begin{gathered}
\Omega(\sigma,\lambda):=-2\dim(\sigma)\gamma-\\ \frac{1}{2}
\sum_{\alpha\in\Delta^{+}(\mathfrak{g}_\C,\mathfrak{a}_\C)}\frac{\Pi(s_{\alpha}\lambda_{\sigma})
}{\Pi(\rho_{M})}\left(\psi(1+\lambda_{\sigma}(H_{\alpha}))+\psi(1-\lambda_{
\sigma}(H_{\alpha}))\right).
\end{gathered}
\end{equation*}\qed
\end{Th}
For regular $\gamma$, the Fourier transform of the distribution $I_L(\gamma, \al)$ was computed in \cite[Theorem 1]{HO}. Later in the paper  he treated non-regular elements as well, for example to obtain \cite[Corollary~on~p.96]{HO}. In this paper we will focus only on regular $\gamma$ for the following two reasons. First, non-regular elements can be  treated in a similar manner. Second, we are most interested in the case when $\Gamma = \PSL(2, \Z \oplus i \Z)$ and $\PSL(2, \Z \oplus \frac{-1+i \sqrt{3}}{2} \Z)$, where all elements, except for the identity, are regular. For the reader's convenience, we first consider the Fourier transform of weighted orbital integrals on $\SO_0(1,3)$ and then proceed to $\SO_0(1,2n+1)$.

\subsection{Fourier transform on $\SO_0(1,3)$} We use
  the notations of \cite{Pf}. For a representation $\sigma \in \widehat{M}$ with highest weight $k_2(\sigma)$ and $\lambda \in \R$ define $\lambda_\sigma \in \mathfrak{h}_\C^*$ by 
 \begin{equation}\label{sugar}
 \lambda_\sigma := i \lambda e_1 + k_2(\sigma) e_2.
 \end{equation}
 Note that the non-identity element $w_0$ of the Weyl group $W$ from Subsection \ref{whatareyoudoingman} acts on $\lambda_\sigma$ as
 \begin{equation}\label{salt}
 w_0 \lambda_\sigma = - i \lambda e_1 - k_2(\sigma) e_2.
 \end{equation}
 Recall that $\Sigma_P^+ = \{e_1 - e_2 \} \cup \{ e_1 + e_2  \}$, and
 \begin{equation}
 \begin{gathered}
 \lambda_\sigma(H_{e_1 - e_2}) = i \lambda - k_2(\sigma), \quad \lambda_\sigma(H_{e_1 + e_2}) =  i \lambda + k_2(\sigma), \\
 \lambda_{w_0 \sigma}(H_{e_1 - e_2}) = - i \lambda + k_2(\sigma), \quad \lambda_{w_0\sigma}(H_{e_1 + e_2}) = - i \lambda - k_2(\sigma).
 \end{gathered}
 \end{equation}
 Let 
 \begin{equation}\label{howcanIbelost}
 \gamma = \left( \begin{smallmatrix}
  1 & 0 & & \\ 0 & 1 & & \\ & & \cos (2 \phi) & \sin (2 \phi) \\ & & -\sin (2 \phi) & \cos (2 \phi)  
  \end{smallmatrix} \right),
 \end{equation} 
 then 
 $$\gamma^{n (e_1 - e_2)} = e^{- 2 i n \phi}, \quad \gamma^{n (e_1 + e_2)} = e^{ 2 i n \phi},$$
 and
 $$ \gamma^{\lambda_\sigma} = e^{2 i \phi \, k_2(\sigma)}, \quad \gamma^{w_0 \lambda_\sigma} = e^{- 2 i \phi \, k_2(\sigma)}.$$
Then \cite[Theorem 1]{HO} reads

 \begin{Th}
\label{fables}
  Let $\gamma \in \Gamma_M(G)$, then 
 \begin{equation}
 I_L(\gamma, \al) = \frac{1}{2 \pi i} \sum_{\sigma \in \widehat{M}} \int_\R \Omega(\gamma, \breve{\sigma}) \Theta_{\sigma,\lambda} (\al) d \lambda,
 \end{equation}
 where
 \begin{equation}\label{Omega_example}
 \begin{gathered}
 \Omega(\gamma, \sigma) = \frac{1}{2} \left[    e^{2 i \phi k_2(\sigma)} \left(   \sum_{n=1}^\infty \frac{e^{2 i n \phi}}{n+i \lambda - k_2(\sigma)} + \sum_{n=1}^\infty \frac{e^{- 2 i n \phi}}{n + i\lambda + k_2(\sigma)} \right) + \right. \\
 \left.    e^{- 2 i \phi k_2(\sigma)} \left(   \sum_{n=1}^\infty \frac{e^{2 i n \phi}}{n-i \lambda + k_2(\sigma)} + \sum_{n=1}^\infty \frac{e^{- 2 i n \phi}}{n-i\lambda - k_2(\sigma)} \right) \right].
 \end{gathered}
 \end{equation}
 \end{Th}
For convenience, let us express $\Omega(\gamma, \sigma)$ in terms of the digamma function. For this denote
$$ b(s,z) = \sum_{n=1}^\infty \frac{z^n}{n+s}.$$

\begin{Lemma}\label{kisa} We have that
 $$b(s, e^{i \frac{2 \pi}{m}}) = \frac{1}{m} \sum_{k=0}^{m-1} e^{\frac{2 \pi i}{m} (m-k)} \psi\left(\frac{s-k}{m}+1\right),$$
 where $\psi$ is the digamma function, $m \in \N$. 
\end{Lemma}
Every $\gamma\in \Gamma_M(P)$ is of finite order by Remark \ref{marinad}, thus Theorem \ref{fables} and Lemma \ref{kisa} immediately imply the following corollary:
\begin{Corollary} For  $\Omega(\gamma, \sigma)$ as in (\ref{Omega_example}) we have:
$$\Omega(\gamma, \sigma) = \sum_{j \in J}  c_j \cdot \phi\big(a_j + i  b_j \cdot \lambda  \big),$$
where $J$ and $a_j$ do not depend on $k_2(\sigma)$; $c_j = O(1)$, $b_j = O(k_2(\sigma))$ as $k_2(\sigma) \to \infty$; and $\Omega(\gamma, \sigma)$ has no pole at $\lambda = 0$. 
\end{Corollary}
\begin{proof}[Proof of Lemma \ref{kisa}]
  We will consider the case when $m = 2$ only. The proof for other $m \in \N$ is done by analogy, but with more technicalities, therefore we omit it. Recall that
 $$ \psi(1+z) = \lim_{n \to \infty} \left(\ln n - \frac{1}{z+1} - \ldots - \frac{1}{z+n}  \right), \quad z\neq -1, -2, \ldots.$$
Hence
 \begin{equation}\label{riegel}
  \psi\left(\frac{s-1}{2}+1\right) = \lim_{n \to \infty}  \underbrace{ \left(  \ln n - \frac{2}{s+1} - \frac{2}{s+3} - \ldots - \frac{2}{s-1+2n} \right) }_{=:A_n(s)},
 \end{equation}
 \begin{equation}\label{droegel}
  \psi\left( \frac{s}{2}+1 \right) = \lim_{n \to \infty} \underbrace{  \left( \ln n - \frac{2}{s+2} - \frac{2}{s+4} - \ldots - \frac{2}{s+2n} \right) }_{=: B_n(s)}.
 \end{equation}
Put
\begin{equation}\label{windy}
 C_n(s) := - \frac{1}{s+1} + \frac{1}{s+2}  -  \ldots + (-1)^n \frac{1}{s+n},
\end{equation}
 then 
 $$b(s,-1) = \lim_{n \to \infty} C_n(s).$$  
It follows from (\ref{riegel}), (\ref{droegel}) and (\ref{windy}) that
  $$ C_{2n} = \frac{A_n(s) - B_n(s)}{2}, \quad C_{2n+1} = \frac{A_{n+1}(s)-B_{n}(s)}{2},$$
  hence
  $$ \lim_{n \to \infty} C_{2n}(s) = \lim_{n \to \infty} C_{2n+1}(s) = \lim_{n\to \infty} \frac{A_n(s) - B_n(s)}{2}$$
  and
  $$ b(s,-1) = \lim_{n \to \infty} C_n = \lim_{n\to \infty} \frac{A_n(s) - B_n(s)}{2} = \frac{\psi \left(  \frac{s}{2} + \frac{1}{2}  \right) - \psi\left( \frac{s}{2}+ 1  \right)}{2}, $$
 that proves Lemma \ref{kisa} for $m=2$. 
 \end{proof}

\subsection{Fourier transform on $\SO_0(1,2n+1)$}
 
Lemma \ref{kisa} together with \cite[Theorem 1]{HO} implies:
\begin{Th}\label{FTorbint}
For every $K$-finite $\alpha\in C^{2}(G)$ one has
\begin{align*}
\mathcal{E}^{cusp}(\alpha)=\sum_{\sigma\in\hat{M}}\int_{\mathbb{R
}}{\Omega^{cusp}(\check{\sigma},-\lambda)\Theta_{\sigma,\lambda}(\alpha)d\lambda},
\end{align*}
where
\begin{align*}
\Omega^{cusp}(\sigma,\lambda):= \sum_{j \in I} c_j \psi(a_j + b_j i \lambda).
\end{align*}
Above $I$ is a finite set; $a_j,b_j \in \R$; $c_j \in \C$; $\psi$ is a digamma function; and $\Omega^{cusp}(\sigma,\lambda)$ is regular at $\lambda = 0$. Let $(\tau_2+m)e_1+\ldots + (\tau_{n+1}+m)e_{n+1}$ be the highest weight of $\sigma$. Then $a_j = O(1)$, $c_j=O(1)$, $|J|=O(1)$ and $b_j=O(m)$ when $m \to \infty$.\qed
\end{Th}

\subsection{Asymptotic expansion of the regularized trace}
In order to define the analytic torsion, we need to know that the regularized trace $\Tr_{reg} e^{-t A_\nu}$ admits certain asymptotic expansion as $t \to +0$. 
For this we need the following lemmas:
\begin{Lemma}\label{pinkyellow}
Let $\phi(t) := \int_{\R} \frac{e^{-t \lambda^2}}{\lambda + c}\; d \lambda$, where $c \neq 0$. Then there exist $a'_j \in \C$ such that 
$$ \phi(t) \sim \sum_{j=0}^\infty a'_j t^{j/2}.$$
\end{Lemma}
\begin{proof}
Note that 
\begin{equation*}
\int_{\R} \frac{e^{-t \lambda^2}}{\lambda + c} d \lambda = \int_{\R} \frac{e^{-t \lambda^2} \lambda}{\lambda^2 - c^2} d \lambda - c \int_{\R} \frac{e^{-t \lambda^2}}{\lambda^2 - c^2} d \lambda = -c e^{-t c^2} \int_{\R} \frac{e^{-t \lambda^2+t c^2}}{\lambda^2 - c^2} d \lambda.
\end{equation*}
and (correcting a mistake in \cite[Lemma 6.6]{MP111}), $$\frac{d}{dt} \int_{\R} \frac{e^{-t \lambda^2+t c^2}}{\lambda^2 - c^2} d \lambda = - \frac{\sqrt{\pi}}{\sqrt{t}} e^{t c^2}.$$
It follows from the previous equation that
\begin{equation*}
\quad \int_{\R} \frac{e^{-t \lambda^2+t c^2}}{\lambda^2 - c^2} d \lambda = - C_1 \cdot \textnormal{erfc}(\sqrt{t})+C_2
\end{equation*}
for some $C_1$ and $C_2$. Expanding $\textnormal{erfc}(\sqrt{t})$ in power series implies Lemma~\ref{pinkyellow}. 
\end{proof}
\begin{Lemma}\label{Lemas2}
Let $\phi_2(t):=\int_{\mathbb{R}}{e^{-t\lambda^2}\psi(a+i\lambda)d\lambda}$, where $a\in(0, 1]$ and $\psi$ is the digamma function.
Then
there exist $a'_j$, $b'_j$, $c'_j \in \C$ such that as $t\to 0$,
there is an asymptotic expansion 
\begin{align*}
\phi_2(t)\sim \sum_{j=0}^\infty a'_jt^{j-1/2}+\sum_{j=0}^\infty b'_j t^{j-1/2}\log{t}
+\sum_{j=0}^\infty c'_jt^{j}.
\end{align*}
\end{Lemma}
\begin{proof}
Follows from the proof of \cite[Lemma 6.7]{MP}
with minor modifications.
\end{proof}
\begin{Corollary}\label{cor3.10}
Let $\phi_3(t):=\int_{\mathbb{R}}{e^{-t\lambda^2}\psi(a+i b \lambda)d\lambda}$, where $a, b \in \R$  and $\psi$ is the digamma function.
Then
there exist $a'_j$, $b'_j$, $c'_j \in \C$ such that as $t\to 0$,
there is an asymptotic expansion 
\begin{align*}
\phi_2(t)\sim \sum_{j=0}^\infty a'_jt^{j-1/2}+\sum_{j=0}^\infty b'_j t^{j-1/2}\log{t}
+\sum_{j=0}^\infty c'_jt^{j}.
\end{align*}
\end{Corollary}
\begin{proof}
Without loss of generosity we can assume that $b>0$ and, moreover, $b=1$. Then Corollary \ref{cor3.10} follows from Lemmas \ref{pinkyellow} and \ref{Lemas2}, taking into account that  $\phi(z+1) = \phi(z)+\frac1z$.
\end{proof}
The main result of the subsection is the following proposition:
\begin{Proposition}
\label{asympbaby}
There exist coefficients $a'_j, b'_j, c'_j$, where $j \in \N$ such that
$$ \Tr_{reg} (e^{-t A_\nu}) \sim \sum_{j=0}^{\infty} a_j t^{j-d/2} + 
\sum_{j=0}^\infty b'_j t^{j-1/2} \log t + \sum_{j=0}^\infty c'_j t^j,$$
as $t \to +0$, where $\textnormal{Tr}_{reg}\left(e^{-tA_\nu}\right)$ is from Theorem \ref{regtrace2}.
\end{Proposition}
\begin{proof}
By Theorem \ref{whattheyknow} it is sufficient to show that the 
summands on the right hand side of (\ref{regtrace2}) admit an asymptotic
expansion as $t~\to~+0$. The summands $I(h^{\nu}_{t}), H(h^{\nu}_{t}), T(h^{\nu}_{t}),\mathcal{I}(h^{\nu}_{t}),J(h^{\nu}_{t})$ were treated in \cite[Proposition 6.9]{MP}. The terms $\mathcal{J}^{cusp}(h_t^\nu)$ and $E(h_t^\nu)$ are treated in the same way as $J(h_t^\nu)$ and $I(h_t^\nu)$, respectively. The asymptotic expansion of 
$\mathcal{E}^{cusp}(h_t^\nu)$ follows from Corollary~\ref{cor3.10} and Theorem~\ref{FTorbint}.
\end{proof}

\section{Analytic torsion}\label{Beweis}
In this section we define the analytic torsion on finite volume orbifolds and prove Theorem \ref{adacupcakeprincess}.
Let $\Ob = \Gamma \bs \HH^{2n+1}$, let $(\rho, V_\rho)$ be a finite-dimensional representation of $\Gamma$ and let $E_\rho \to \Ob$ be the associated flat orbibundle. 

Let us specify to the case where $\rho = \tau|_\Gamma$ is the restriction to $\Gamma$ of a finite-dimensional irreducible representation $\tau$ of $G$. In this case $E_\rho$ can be equipped with a distinguished metric which is unique up to scaling. Namely, $E_\rho$ is canonically isomorphic to the locally homogeneous orbibundle~$E_\tau$ associated to $\tau|_K$  (by analogy with \cite[Proposition 3.1]{MaMu}). Moreover, there exists a unique up to scaling inner product~$\langle \cdot, \cdot \rangle$ on $V_\rho$ such that
\begin{enumerate}
\item $\langle \tau(Y) u, v \rangle = - \langle u, \tau(Y) v \rangle, \quad  Y \in \mathfrak{k}$,
\item $\langle \tau(Y) u, v \rangle =  \langle u, \tau(Y) v \rangle, \quad  Y \in \mathfrak{p}$,
\end{enumerate}
for all $u,v \in V_\rho$. Note that $\tau|_K$ is unitary with respect to this inner product, hence it induces a unique up to scaling metric~$h$ on $E_\tau$.
\begin{definition}
\label{admi}
Such a metric on $E_\tau$ is called admissible.
\end{definition}
From now on fix an admissible metric $h$. 
Let $\Delta_p(\tau)$ be the Hodge-Laplacian on $\Lambda^p(\Ob, E_\tau)$ with respect to $h$. 
\begin{Lemma}\cite[(6.9)]{MaMu}
\label{lemma_chat}
One has 
$$ \Delta_p(\tau)=-\Omega+\tau(\Omega) \Id$$
for $$\tau(\Omega)=\sum_{j=1}^{n+1}(k_j(\tau)+\rho_j)^2-\sum_{j=1}^{n+1} \rho_j^2,$$
where $k_1(\tau)e_1+\ldots+k_{n+1}(\tau)e_{n+1}$ is the highest weight of $\tau$ and $\rho_j$ is from (\ref{halfsummy}).
\end{Lemma}
\begin{remark}
Compare with the differential operator $A_\nu$ from Definition \ref{yanenavizhu}.
\end{remark}
In order to define the spectral zeta function we need to study the asymptotic behavior of $\Tr_{reg} e^{-t\Delta_p(\tau)}$ as $t \to 0$ and $t \to \infty$. It follows from Lemma \ref{lemma_chat} and Proposition \ref{asympbaby} that  there exist coefficients $a'_j, b'_j, c'_j$, $j \in \N$ such that
\begin{equation}
\label{assnumberone}
 \Tr_{reg} (e^{-t\Delta_p(\tau)}) \sim \sum_{j=0}^{\infty} a_j t^{j-d/2} + 
\sum_{j=0}^\infty b'_j t^{j-1/2} \log t + \sum_{j=0}^\infty c'_j t^j,
\end{equation}
as $t \to +0$. As in \cite[(7.10)]{MP} it follows from (\ref{regtrace2})  that
\begin{equation}
\label{assnumbertwo}
\Tr_{reg}\left(e^{-t\Delta_p(\tau)}\right)\sim h_p(\tau)+\sum_{j=1}^\infty c_j 
t^{-j/2},\quad t\to\infty,
\end{equation}
where $h_p(\tau) = \dim(\ker \Delta_p(\tau) \cap L^2)$.
\begin{definition}
The spectral zeta function is defined as:
$$ \zeta_p(s;\tau):=\frac{1}{\Gamma(s)} \int_0^1+\int_1^\infty t^{s-1} \Tr_{reg} \left( e^{-t\Delta_p(\tau)} - h_p(\tau)\right) dt.$$
\end{definition}
By (\ref{assnumberone}) and (\ref{assnumbertwo}) both integrals admit meromorphic continuation to~$\C$ that is regular at $s=0$ and hence we can define:
\begin{definition}
\label{defineanal}
The analytic torsion $T_\Ob(\tau)$ associated with a flat vector bundle $E_\tau$ equipped with the admissible metric from Definition \ref{admi}, is defined as
$$ T_\Ob(\tau)=\prod_{p=0}^{2n+1} \exp\left. \left( -\frac{d}{ds} \zeta_p(s;\tau)\right|_{s=0} \right)^{(-1)^{p+1} \cdot p/2}.$$
\end{definition}
Let
$$ K(t,\tau):=\sum_{p=0}^p (-1)^{p} \, p \, \Tr_{reg} e^{-t \Delta_p(\tau)}$$
Note that if $h_p(\tau)=0$, the analytic torsion is given by:
\begin{equation}
\label{expressyourself}
\log{T_{\Ob}(\tau)}=\frac{1}{2}\frac{d}{ds}\biggr|_{s=0}\left(\frac{1}{\Gamma(s)}
\int_{0}^{\infty}t^{s-1}K(t,\tau)\,dt\right).
\end{equation}
\begin{remark}
If $\tau=\tau(m)$, then $h_p(\tau(m))=0$ for sufficiently large $m$ \cite[Lemma 7.3]{MP}. 
\end{remark}
Let $\w{E}_{\nu_p(\tau)} := G \times_{\nu_p(\tau)} \Lambda^p \mathfrak{p}^* \otimes V_\tau$, where 
$$\nu_p(\tau):=\Lambda^p \Ad^* \otimes \tau: K \mapsto GL(\Lambda^p \mathfrak{p}^* \otimes V_\tau)$$ and let $\w{\Delta}_p(\tau)$ be the lift of $\Delta_p(\tau)$ to $C^\infty(\HH^{2n+1}, \w{E}_{\nu_p(\tau)})$. Denote by $H_t^{\tau,p}:G \mapsto \End(\Lambda^p \mathfrak{p}^* \otimes V_\tau)$ the convolution kernel of $e^{-t \w{\Delta}_p(\tau)}$ as in \cite[p.~16]{MP111}. Let 
\begin{equation}\label{ananas}
 h_t^{\tau,p}(g):=\tr\, H_t^{\tau,p}(g),
\end{equation}
where $\tr$ denotes the trace in $\End(V_\nu)$. 
Put
$$ k_t^\tau(g) := e^{-t \tau(\Omega)}\sum_{p=1}^{2n+1} (-1)^p  \, p \, h_t^{\tau,p}(g).$$
It follows from Theorem \ref{whattheyknow} that 
\begin{equation}
\label{sailormoonfight}
 K(t,\tau)=\left(I+H+T+\mathcal{I}+J+E+\mathcal{E}^{cusp}+\mathcal{J}^{cusp} \right)(k_t^\tau).
\end{equation}
By (\ref{expressyourself}) and (\ref{sailormoonfight}) in order to study the analytic torsion $T_{\Ob}(\tau)$ we need to study the Mellin transform of the right hand side of (\ref{sailormoonfight}) at zero. First we express $k_t^\tau(g)$ in a more convenient way:
\begin{Proposition}\cite[Proposition 8.2, (8.13)]{MP}
\label{easyeasygogo}
For $k=0,\ldots,n$ let 
\begin{equation}
\label{hatehate}
 \lambda_{\tau,k}=\tau_{k+1}+n-k,
\end{equation}
 $\sigma_{\tau,k}$ be as in Definition \ref{dertttt}, and $h_t^{\sigma_{\tau,k}}$ be as in \cite[(8.8)]{MP}. Then 
\begin{equation}
k_t^\tau=\sum_{k=0}^{n}(-1)^{k+1}e^{-t\lambda_{\tau,k}^2}h_t^{\sigma_{\tau,k}}.
\end{equation} 
For a principal series 
representation $\pi_{\sigma',\lambda}$, $\lambda\in\mathbb{R}$, the Fourier transform is:
\begin{align}\label{FTSup}
\Theta_{\sigma',\lambda}(h^{\sigma}_{t})=e^{-t\lambda^{2}} \quad
\text{for $\sigma'\in\{\sigma, w_{0}\sigma\}$};\qquad
\Theta_{\sigma',\lambda}(h^{\sigma}_{t})=0, \quad\text{otherwise}. 
\end{align}
\end{Proposition}
\subsection{Asymptotic behavior of the analytic torsion}
From now on let $\tau(m)$ be the ray of representations of $G$ from in Definition~\ref{mimirep}. Denote by the same symbol its restriction to $\Gamma$. Let $\mathcal{M}I(\tau(m))$, $\mathcal{M}H(\tau(m))$, $\mathcal{M}T(\tau(m))$, $\mathcal{M}\mathcal{I}(\tau(m))$ $\mathcal{M}J(\tau(m))$ be as in \cite{MP} and $\mathcal{M}E(\tau(m))$ be as in \cite{Fe2}. Roughly speaking, they equal the value of Mellin transforms at zero of the corresponding terms in the right hand side of (\ref{sailormoonfight}) with $\tau=\tau(m)$.
\begin{Th} There exists a constant $C$ such that for $m$ sufficiently large one has
\begin{equation*}
\begin{gathered}
\mathcal{M}I(t, \tau(m)) = C(n)\vol(X)m\dim{\tau(m)}+O(m^{\frac{n(n+1)}{2}}),\\
|\mathcal{M}H(\tau(m))|\leq Cm^{\frac{n(n-1)}{2}}, \quad 
|\mathcal{M}T(\tau(m))|\leq Cm^{\frac{n(n+1)}{2}},\\
|\mathcal{M}J(\tau(m))|\leq Cm^{\frac{n(n+1)}{2}}\log{m}, \quad 
\left|\mathcal{M}\mathcal{I}(\tau(m))\right|\leq Cm^{\frac{n(n+1)}{2}},\\
|\mathcal{M}E(t, \tau(m))| \le C m^{\frac{n(n-1)}{2}}.
\end{gathered}
\end{equation*}
\end{Th}
\begin{proof}
Follows from \cite[Propositions 10.1, 10.3, 10.4, 10.10, 10.14]{MP} and \cite{Fe2}.
\end{proof}

\begin{Th} There exists a constant $C$ such that for $m$ sufficiently large one has
\begin{equation*}
|\mathcal{M}J^{cusp}(\tau(m))|\leq Cm^{\frac{n(n+1)}{2}}\log{m}.
\end{equation*}
\end{Th}
\begin{proof}
Follows with minor modifications from \cite[Proposition 10.14]{MP}, keeping in mind Remark \ref{apocalyptiss}.
\end{proof}

It follows from Corollary \ref{cor3.10} that $\mathcal{E}^{cusp}(k_t^{\tau(m)})$ admits an asymptotic expansion 
$$\mathcal{E}^{cusp}(k_t^{\tau(m)}) \sim \sum_{j=0}^\infty a'_jt^{j-1/2}+\sum_{j=0}^\infty b'_j t^{j-1/2}\log{t}
+\sum_{j=0}^\infty c'_jt^{j}$$
for some $a'_j$, $b'_j$, $c'_j \in \C$.
Moreover, $\mathcal{E}^{cusp}(k_t^{\tau(m)}) = O(e^{-t m^2})$ as $m \to \infty$. Hence 
$$M\mathcal{E}^{cusp}(s; \tau(m)) := \int_0^\infty t^{s-1} \mathcal{E}^{cusp}(k_t^{\tau(m)}) \;dt$$
converges for $\RRe(s) > \frac{d-1}{2}$, admits a meromorphic continuation to $\C$ and has at most simple pole at $s=0$. Denote 
$$ M\mathcal{E}^{cusp}(\tau(m)):= \frac{d}{ds} \left. \frac{M\mathcal{E}^{cusp}(s; \tau(m))}{\Gamma(s)} \right|_{s=0}.$$

\begin{Th}\label{wagnerreloaded} There exists a constant $C$ such that for $m$ sufficiently large one has
\begin{equation*}
|\mathcal{M}\mathcal{E}^{cusp}(\tau(m))|\leq C \cdot m  \log(m).
\end{equation*}
\end{Th}
In order to prove Theorem \ref{wagnerreloaded} we need the following technical lemmas.
\begin{Lemma}\label{22:00}
For $c\in(0,\infty)$, $s\in\C$, $\RRe(s)>0$, $e_j, d_j > 0$
let
\begin{align*}
\zeta_{c}(s):=\frac{1}{\pi}\int_{0}^{\infty}{t^{s-1}e^{-tc^2}\int_{\infty}
{\frac{e^{
-tz^{2}}}{i e_j z+d_j}\; dz}\;dt}.
\end{align*}
 Then $\zeta_{c}(s)$ has a meromorphic continuation to
$\C$ with a simple pole at~0. Moreover, one has
\begin{align*}
\frac{d}{ds}\biggr|_{s=0}\frac{\zeta_c(s)}{\Gamma(s)}=-\frac{2}{e_j}\log{\left(c+d_j/e_j\right)}.
\end{align*}
\end{Lemma}
\begin{proof}
Follows with minor modifications from \cite[Lemma 10.5]{MP}\label{Ratcontr}.
\end{proof}

\begin{Lemma}\label{Gammacontr}
Let $c\in\R^+$, $s\in\C$, $\RRe(s)>1/2$, $a_j, b_j > 0$.
Define
\begin{align*}
\tilde{\zeta}_c(s):=\frac{1}{\pi}\int_{0}^{\infty}{t^{s-1}e^{-tc^2}\int_{
\mathbb{R}}{e^{
-t\lambda^{2
}}\psi\left(a_j+i b_j \lambda \right)d\lambda}\;dt}.
\end{align*}
Then 
$\tilde{\zeta}_c(s)$ has a meromorphic continuation to
$s\in\C$ with at most a simple pole at $s=0$. Moreover, there exist a constant
$C(\psi)$ which is independent of $c$, $a_j$ and $b_j$ such that
\begin{align*}
\frac{d}{ds}\biggr|_{s=0}\frac{\tilde{\zeta}_c(s)}{\Gamma(s)}=-\frac{2}{b_j}\log
\Gamma(a_j+c b_j)+C(\psi).
\end{align*}
\end{Lemma}
\begin{proof}
Follows from \cite[Lemma 10.6]{MP} with minor modifications.
\end{proof}
\begin{proof}[Proof of Theorem \ref{wagnerreloaded}]
Let 
\[
\mathcal{M}\mathcal{E}^{cusp}(s;\sigma_{\tau(m),k}):=\int_0^\infty t^{s-1} 
e^{-t\lambda_{\tau(m),k}^2}\mathcal{E}^{cusp}(h_t^{\sigma_{\tau(m),k}})\;dt.
\]
As above it follows 
that the integral converges for $\RRe(s)>(d-2)/2$ and
admits a meromorphic continuation to $\C$ with at most a simple
pole  at $s=0$. By \cite[Proposition 8.2]{MP},
\begin{align*}
\mathcal{M}\mathcal{E}^{cusp}(\tau(m))=\sum_{k=0}^n(-1)^{k+1}\frac{d}{ds}\biggr|_{s=0}
\frac{\mathcal{M}\mathcal{E}^{cusp}
(s;\sigma_{\tau(m),k})}{\Gamma(s)}.
\end{align*}
In order to prove Theorem \ref{wagnerreloaded} it suffices to consider 
\begin{equation}\label{BIOLOGICAL_TRASH}
 \frac{d}{ds}\biggr|_{s=0}
\frac{\mathcal{M}\mathcal{E}^{cusp}
(s;\sigma_{\tau(m),k})}{\Gamma(s)} = \frac{d}{ds} \left. 
\left(  \int_0^\infty t^{s-1} e^{-t \lambda_{\tau(m),k}^2} 
\mathcal{E}^{cusp} (h_t^{\sigma_{\tau(m),k}})  dt \right) \right|_{s=0}
\end{equation}
as $m \to \infty$.
By Theorem \ref{FTorbint}  we can rewrite the left hand side of (\ref{BIOLOGICAL_TRASH}) as
$$  \frac{d}{ds}\biggr|_{s=0} \int_0^\infty t^{s-1} e^{-t \lambda_{\tau(m),k}^2} \int_{\R} \left( \Omega(\sigma_{\tau(m),k}, \lambda) + \Omega(w_0 \sigma_{\tau(m),k}, \lambda)\right) e^{-t \lambda^2} \;d \lambda \;dt,$$
where 
$$\Omega(\sigma_{\tau(m),k}, \lambda)=\sum_{j \in J} c_j \cdot \psi(a_j + i \lambda b_j)$$
for some $a_j\in \R$, growing not faster than linearly in $m$; $c_j$, that is bounded by a constant as $m$ grows; $b_j\in \R$ not depending on $m$; and $J$ a finite set, not depending on $m$ as well. 
Note that 
$$ \int_{\R} \psi(a_j + i b_j \lambda) e^{-t \lambda^2} d \lambda = \int_{\R} \psi(a_j - i  b_j \lambda) e^{-t \lambda^2} d \lambda,$$
hence we can assume that all $b_j > 0$. As $\psi(z+1)=\psi(z)+\frac{1}{z}$, we can  we can rewrite the left hand side of (\ref{BIOLOGICAL_TRASH}) as
\begin{equation}\label{22:02}
\frac{d}{ds}\biggr|_{s=0} \int_0^\infty t^{s-1} e^{-t \lambda_{\tau(m),k}^2} \int_{\R} \left( \sum_{j \in J'} c_j \cdot \psi(a_j + i \lambda b_j) + \sum_{j \in J''} \frac{1}{i e_j \lambda + d_j}\right) e^{-t \lambda^2} \;d\lambda\;dt.
\end{equation}
Above $a_j, b_j, c_j, d_j > 0$, $J$ and $J'$ are finite sets; $b_j$, $e_j$ and $|J'|$ does not depend on $m$; $a_j$, $d_j$ and $|J''|$ grow not faster than linearly in $m$; $c_j$ is bounded by a constant as $m$ grows. By (\ref{22:02}), Lemmas \ref{22:00} and \ref{Gammacontr}
\begin{equation}\label{22:16}
\begin{gathered}
 \frac{d}{ds}\biggr|_{s=0}
\frac{\mathcal{M}\mathcal{E}^{cusp}
(s;\sigma_{\tau(m),k})}{\Gamma(s)} = - \sum_{j \in J'} \left( \frac{2 c_j}{b_j} \log \Gamma(a_j + b_j \lambda_{\tau(m),k}) + c_j \cdot C_j(\psi)\right) - \\ \sum_{j \in J''} \frac{2 c_j}{e_j} \log (d_j/e_j + \lambda_{\tau(m),k}).
\end{gathered}
\end{equation}
By (\ref{hatehate}), 
$$a_j + b_j \lambda_{\tau(m),k}=O(m),$$
hence 
\begin{equation}\label{22:12}
\log \Gamma(a_j + b_j \lambda_{\tau(m),k}) = O(m \cdot \log m) 
\end{equation}
by the Stirling's formula. On the other hand, 
\begin{equation}\label{22:15}
\sum_{j \in J''} \frac{2}{e_j} \log (d_j/e_j + \lambda_{\tau(m),k}) = O(m \cdot \log m).
\end{equation}
Putting together (\ref{BIOLOGICAL_TRASH}), (\ref{22:16}), (\ref{22:12}) and (\ref{22:15}) proves Theorem \ref{adacupcakeprincess}.
\end{proof}

\bibliography{foo}{}
\bibliographystyle{alpha}
\end{document}